\documentclass{amsart}
\usepackage{amsmath, amssymb, amsthm, mathrsfs}

\theoremstyle{plain} 
 \newtheorem{thm}{Theorem}[section]
 \newtheorem{lem}[thm]{Lemma}
 
 \newtheorem{prop}[thm]{Proposition}
 \newtheorem{claim}[thm]{Claim}
\theoremstyle{definition}
  \newtheorem{defn}[thm]{Definition}

\theoremstyle{remark}
  \newtheorem{rem}[thm]{Remark}

\newcommand{\aut}{{\rm Aut}}
\newcommand{\bs}{{\rm BS}}

\newcommand{\R}{\mathbb{R}}

\newcommand{\cal}{\mathcal}
\newcommand{\calb}{\mathcal{B}}
\newcommand{\N}{\mathbb{N}}
\newcommand{\Z}{\mathbb{Z}}
\newcommand{\calr}{\mathcal{R}}
\newcommand{\ci}[2]{\cite[#1]{#2}}
\renewcommand{\c}{\curvearrowright}

\begin{document}

\title[Stability in orbit equivalence]{Stability in orbit equivalence for Baumslag-Solitar groups and Vaes groups}
\author{Yoshikata Kida}
\address{Department of Mathematics, Kyoto University, 606-8502 Kyoto, Japan}
\email{kida@math.kyoto-u.ac.jp}
\date{Februay 18, 2013}

\begin{abstract}
A measure-preserving action of a discrete countable group on a standard probability space is called stable if the associated equivalence relation is isomorphic to its direct product with the ergodic hyperfinite equivalence relation of type ${\rm II}_1$.
We show that any Baumslag-Solitar group has such an ergodic, free and stable action.
It follows that any Baumslag-Solitar group is measure equivalent to its direct product with any amenable group.
The same property is obtained for the inner amenable groups of Vaes.
\end{abstract}

\maketitle


\section{Introduction}\label{sec-int}

We mean by a {\it p.m.p.} action of a discrete countable group a measure-preserving action of the group on a standard Borel space equipped with a probability measure, where ``p.m.p." stands for ``probability-measure-preserving''.
A p.m.p.\ action of a discrete countable group is called {\it stable} if the associated equivalence relation is isomorphic to its direct product with the ergodic hyperfinite equivalence relation of type ${\rm II}_1$.
Due to Connes-Feldman-Weiss \cite{cfw} and Ornstein-Weiss \cite{ow}, any ergodic, free and p.m.p.\ action of any infinite amenable group gives rise to the ergodic hyperfinite equivalence relation of type ${\rm II}_1$.
It is a challenging problem to decide whether a given group has an ergodic, free, p.m.p.\ and stable action unless the group virtually has a direct summand which is infinite and amenable.

Jones-Schmidt \cite{js} characterized stability of ergodic p.m.p.\ actions in terms of asymptotically central sequences.
In \cite[Example 4.4]{js}, they also noticed that for any collection of countably infinitely many, non-trivial discrete countable groups, $\{ G_n\}_n$, and for any ergodic, free and p.m.p.\ action $G_n\c (X_n, \mu_n)$, the product action of the direct sum, $\oplus_n G_n\c \prod_n(X_n, \mu_n)$, is stable.

On the other hand, Zimmer \cite{zim-prod} obtained certain indecomposability results on equivalence relations arising from semisimple Lie groups.
Adams \cite{adams} showed that for any ergodic, free and p.m.p.\ action of a non-elementary word-hyperbolic group, the associated equivalence relation cannot be written as the direct product of two discrete measured equivalence relations of type ${\rm II}_1$.
Such indecomposability is also obtained for any equivalence relation with its cost more than $1$ or its first $\ell^2$-Betti number positive in \cite{gab-c} and \cite{gab-l}, and for the equivalence relation associated with any action of the mapping class group of a surface in \cite{kida-mcg}.
More strongly, the von Neumann algebras associated with various group actions are shown to be prime in \cite{ch}, \cite{cs}, \cite{hv}, \cite{is}, \cite{oz} and \cite{popa-gap}.

For two integers $p$, $q$ with $1\leq |p|\leq |q|$, the group with the presentation, 
\[\bs(p, q)=\langle\, a, t\mid ta^pt^{-1}=a^q\,\rangle,\]
is called the {\it Baumslag-Solitar group}.
The group $\bs(p, q)$ is amenable if and only if $|p|=1$.
If $2\leq |p|=|q|$, then $\bs(p, q)$ has a finite index subgroup isomorphic to the direct product of the infinite cyclic group $\Z$ with a non-abelian free group of finite rank.
It readily follows that there exists a free and stable action of $\bs(p, q)$ unless $2\leq |p|<|q|$.
If $2\leq |p|<|q|$, then no finite index subgroup of $\bs(p, q)$ has an infinite, amenable and normal subgroup.
This is proved through the action of $\bs(p, q)$ on the Bass-Serre tree and its boundary, as discussed in the first paragraph of \cite[Appendix B]{kida-bs}.
We could expect Adams' argument in \cite{adams} applicable to $\bs(p, q)$ because the action of a word-hyperbolic group on its compactification plays an important role in his proof.
By contrast, applying Jones-Schmidt's characterization of stable actions, we show the following:

\begin{thm}\label{thm-stable}
Let $p$ and $q$ be integers with $2\leq |p|<|q|$.
Then $\bs(p, q)$ has an ergodic, free, p.m.p.\ and stable action.
In particular, for any amenable, discrete and countable group $A$, the group $\bs(p, q)$ is measure equivalent to the direct product $A\times \bs(p, q)$.
\end{thm}

Any discrete countable group having an ergodic, free, p.m.p.\ and stable action is inner amenable (see \cite[Proposition 4.1]{js}).
Theorem \ref{thm-stable} therefore implies that any Baumslag-Solitar group is inner amenable.
The latter was proved by Stalder \cite{stalder}.
We should compare Theorem \ref{thm-stable} with Fima's result in \cite{fima} that the von Neumann algebra of $\bs(p, q)$ with $2\leq |p|<|q|$ is prime, is not solid and has no Cartan subalgebra.
After posting the first draft of this paper on the arXiv, we were informed by Narutaka Ozawa that the von Neumann algebra of $\bs(p, q)$ has property Gamma.
This also implies that $\bs(p, q)$ is inner amenable.

Let $\mathbb{Q}$ denote the field of rational numbers.
For a set of prime numbers, $S$, we define $\Z_S$ as the subring of $\mathbb{Q}$ generated by all $1/s$ with $s\in S$.
As an application of Theorem \ref{thm-stable}, we obtain the following:

\begin{thm}\label{thm-s-int}
Let $p$ and $q$ be integers with $2\leq |p|<|q|$.
Let $S$ be a set of prime numbers dividing neither $p$ nor $q$.
We denote by $\alpha$ the isomorphism from $p\Z_S$ onto $q\Z_S$ multiplying by $q/p$.
Then the HNN extension of $\Z_S$ relative to $\alpha$ is measure equivalent to $\bs(p, q)$.
\end{thm}

In \cite[Problem 4.2]{js}, Jones-Schmidt asked whether any inner amenable group has an ergodic, free, p.m.p.\ and stable action.
Recently, Vaes \cite{vaes} discovered an inner amenable group $G$ whose von Neumann algebra $LG$ is a factor and does not have property Gamma.
This solved a longstanding problem posed by Effros \cite{effros}.
In particular, $LG$ is not isomorphic to its tensor product with the hyperfinite ${\rm II}_1$ factor.
The Vaes group $G$ could therefore be a candidate of a counterexample to the above Jones-Schmidt's question, whereas we show the following:

\begin{thm}\label{thm-v}
The Vaes group $G$ has an ergodic, free, p.m.p.\ and stable action.
In particular, for any amenable, discrete and countable group $A$, the group $G$ is measure equivalent to the direct product $A\times G$.
\end{thm}

The construction of a stable action of the group $G$ is fairly similar to that for Baumslag-Solitar groups.

This paper is organized as follows.
In Section \ref{sec-pre}, we review the characterization of stable actions due to Jones-Schmidt, and introduce the notation and terminology employed throughout the paper.
Theorem \ref{thm-stable} is proved through Sections \ref{sec-sol} and \ref{sec-s}.
In Section \ref{sec-sol}, focusing on a certain solvable quotient of $\bs(p, q)$, we construct its p.m.p.\ action based on the odometer action of $\Z$.
In Section \ref{sec-s}, using this action, we obtain an ergodic, free, p.m.p.\ and stable action of $\bs(p, q)$.
As a by-product, we also find such a stable action of the normal subgroup of $\bs(p, q)$ generated by $a$.
In the case of $p=1$, although $\bs(1, q)$ is amenable, our construction of stable actions is still available, and reduces to the construction in Section \ref{sec-sol}.
This will help us to get an intuition for a general case.
In Sections \ref{sec-sdp} and \ref{sec-v}, Theorems \ref{thm-s-int} and \ref{thm-v} are proved, respectively.


\section{Preliminaries}\label{sec-pre}

Let $\N$ denote the set of non-negative integers.
Let $\Z_-$ and $\Z_+$ denote the set of negative integers and the set of positive integers, respectively.

\subsection{Jones-Schmidt's characterization of stability}

We mean by a {\it standard probability space} a standard Borel space equipped with a probability measure.
All relations among Borel sets and maps that appear in the paper are understood to hold up to sets of measure zero, unless otherwise stated.

Let $(X, \mu)$ be a standard probability space.
We denote by $\calb_X$ the $\sigma$-algebra of Borel subsets of $X$.
Let $\calr$ be an ergodic, discrete measured equivalence relation on $(X, \mu)$ of type ${\rm II}_1$.
For $x\in X$, let $\calr x$ denote the equivalence class of $\calr$ containing $x$.
We define $[\calr]$ as the {\it full group} of $\calr$, that is, the group of Borel automorphisms $U$ of $(X, \mu)$ with $Ux\in \calr x$ for a.e.\ $x\in X$.

A sequence $\{ B_j\}_{j\in \N}$ in $\calb_X$ is called an {\it asymptotically invariant (a.i.) sequence} for $\calr$ if for any $U\in [\calr]$, we have $\lim_j\mu(UB_j\bigtriangleup B_j)=0$.
An a.i.\ sequence $\{ B_j\}_{j\in \N}$ for $\calr$ is called {\it trivial} if we have $\lim_j\mu(B_j)(1-\mu(B_j))=0$.

A sequence $\{ U_j\}_{j\in \N}$ in $[\calr]$ is called an {\it asymptotically central (a.c.) sequence} for $\calr$ if the following two conditions hold:
\begin{enumerate}
\item[(a)] For any $B\in \calb_X$, we have $\lim_j\mu(U_jB\bigtriangleup B)=0$.
\item[(b)] For any $V\in [\calr]$, we have $\lim_j\mu(\{\, x\in X\mid U_jVx\neq VU_jx\,\})=0$.
\end{enumerate}
An a.c.\ sequence $\{ U_j\}_{j\in \N}$ for $[\calr]$ is called {\it trivial} if we have $\lim_j\mu(U_jB_j\bigtriangleup B_j)=0$ for any a.i.\ sequence $\{ B_j\}_{j\in \N}$ for $\calr$.

We mean by a discrete group a discrete and countable group.
Let $G$ be a discrete group, and let $G\c (X, \mu)$ be an ergodic p.m.p.\ action.
We define $\calr(G\c (X, \mu))$ as the discrete measured equivalence relation associated with the action $G\c (X, \mu)$, that is,
\[\calr(G\c (X, \mu))=\{\, (gx, x)\in X\times X\mid g\in G,\ x\in X\,\},\]
which is often denoted by $\calr(G\c X)$ if $\mu$ is understood from the context.
An a.i.\ sequence for $\calr(G\c X)$ is also called an {\it asymptotically invariant (a.i.) sequence} for the action $G\c (X, \mu)$.
Similarly, an a.c.\ sequence for $\calr(G\c X)$ is also called an {\it asymptotically central (a.c.) sequence} for the action $G\c (X, \mu)$.
We note that a sequence $\{ B_j\}_{j\in \N}$ in $\calb_X$ is a.i.\ for the action $G\c (X, \mu)$ if and only if for any $g\in G$, we have $\lim_j\mu(gB_j\bigtriangleup B_j)=0$.
We also note that a sequence $\{ U_j\}_{j\in \N}$ in $[\calr(G\c X)]$ is a.c.\ for the action $G\c (X, \mu)$ if and only if it satisfies condition (a) in the definition of an a.c.\ sequence and the following condition:
\begin{enumerate}
\item[(c)] For any $g\in G$, we have $\lim_j\mu(\{\, x\in X\mid U_jgx\neq gU_jx\,\})=0$.
\end{enumerate}
These remarks are noticed in \cite[Section 2]{js} and \cite[Remark 3.3]{js}, respectively.

\begin{lem}\label{lem-ai}
Let $\calr$ and $\cal{S}$ be ergodic, discrete measured equivalence relations of type ${\rm II}_1$, on standard probability spaces $(X, \mu)$ and $(Y, \nu)$, respectively.
Let $\pi \colon X\to Y$ be a Borel map with $\pi_*\mu=\nu$ and $\pi(\calr x)\subset \cal{S}\pi(x)$ for a.e.\ $x\in X$.
If $\{ B_j\}_{j\in \N}$ is an a.i.\ sequence for $\cal{S}$, then $\{ \pi^{-1}(B_j)\}_{j\in \N}$ is an a.i.\ sequence for $\calr$.
\end{lem}

\begin{proof}
For $j\in \N$, we set $A_j=\pi^{-1}(B_j)$.
Pick $g\in [\calr]$ and $\varepsilon >0$.
It is enough to show that $\mu(gA_j\setminus A_j)<\varepsilon$ for any sufficiently large $j\in \N$.
By \cite[Theorem 1]{fm}, we can find a discrete group $H$ and a p.m.p.\ action of $H$ on $(Y, \nu)$ such that the associated equivalence relation is equal to $\cal{S}$.
Let $\{ h_n\}_{n\in N}$ be an enumeration of all elements of $H$ with $N$ a countable set.
There exists a countable Borel partition $X=\bigsqcup_{n\in N}X_n$ such that for any $n\in N$ and a.e.\ $x\in X_n$, we have $\pi(gx)=h_n\pi(x)$.
Let $F$ be a finite subset of $N$ with $\mu\left(X\setminus \bigcup_{n\in F}X_n\right)<\varepsilon /2$.
Choose a number $J\in \N$ such that for any $n\in F$ and any $j\in \N$ with $j\geq J$, we have $\nu(h_nB_j\bigtriangleup B_j)<\varepsilon /(2|F|)$.
For any $n\in N$ and $j\in \N$, the inclusion
\begin{align*}
g(A_j\cap X_n)\setminus A_j&\subset \{\, x\in X\mid \pi(g^{-1}x)\in B_j,\ h_n\pi(g^{-1}x)=\pi(x)\not\in B_j\,\}\\
&\subset g\pi^{-1}(B_j\setminus h_n^{-1}B_j)
\end{align*}
holds.
For any $j\in \N$ with $j\geq J$, we therefore obtain the inequality
\[
\mu(gA_j\setminus A_j)\leq \mu\left(X\setminus \bigcup_{n\in F}X_n\right)+\sum_{n\in F}\mu(g(A_j\cap X_n)\setminus A_j)<\varepsilon.
\]
The lemma is proved.
\end{proof}

Let $\calr_0$ be the ergodic hyperfinite equivalence relation on a standard probability space $(X_0, \mu_0)$ of type ${\rm II}_1$.
An ergodic, discrete measured equivalence relation $\calr$ on a standard probability space $(X, \mu)$ is called {\it stable} if $\calr$ is isomorphic to the direct product $\calr \times \calr_0$ on $(X\times X_0, \mu \times \mu_0)$ defined as
\[\calr\times \calr_0=\{\, ((x, x_0), (y, y_0))\in (X\times X_0)^2\mid (x, y)\in \calr,\ (x_0, y_0)\in \calr_0\,\}.\]
An ergodic p.m.p.\ action $G\c (X, \mu)$ is called {\it stable} if $\calr(G\c X)$ is stable.
Jones-Schmidt obtained the following characterization of stability.

\begin{thm}[\ci{Theorem 3.4}{js}]\label{thm-js}
An ergodic discrete measured equivalence relation of type ${\rm II}_1$ is stable if and only if it has a non-trivial a.c.\ sequence.
\end{thm}


\subsection{Abelian groups associated with integers}

Fix an integer $l$ with $l\geq 2$.
Put $E=\Z$.
We define a ring $E_l$ as the projective limit $\varprojlim E/l^nE$, which is compact and unital.
We have an embedding of $E$ into $E_l$ through the canonical projection from $E$ onto $E/l^nE$.
Each element $x$ of $E_l$ is uniquely written as the formal sum $x=\sum_{i=0}^{\infty}x_il^i$ with $x_i\in \{ 0, 1,\ldots, l-1\}$ for any $i\in \N$.

The additive group $E_l$ is torsion-free.
In fact, if $k$ is a positive integer and $a$ is an element of $E_l$ with $ka=0$, then choosing a sequence $\{ m_i\}_{i\in \N}$ in $E$ approaching $a$, we obtain the sequence $\{ km_i\}_{i\in \N}$ in $E$ approaching $0$.
It follows that for any $n\in \N$, there exists a number $I\in \N$ such that for any $i\in \N$ with $i\geq I$, the number $km_i$ is divisible by $l^n$.
This is equivalent to that for any $n\in \N$, there exists a number $J\in \N$ such that for any $i\in \N$ with $i\geq J$, the number $m_i$ is divisible by $l^n$.
We therefore have $a=0$.
The claim is proved.

Let $R_l$ denote the ring of fractions of $E_l$ by the multiplicative subset $S=\{\, l^n\in E\mid n\in \N\,\}$.
The ring $R_l$ consists of equivalence classes of all elements in $E_l\times S$, where two elements $(a, s), (b, t)\in E_l\times S$ are equivalent if and only if there exists an element $u\in S$ with $(at-bs)u=0$.
The equivalence class of $(a, s)\in E_l\times S$ is denoted by $a/s$. 
Using that the additive group $E_l$ is torsion-free, we can show that $E_l$ is naturally a subring of $R_l$.
The subring of $R_l$ generated by $1/l$ is isomorphic to the ring $\Z[1/l]$.

Suppose that we have two coprime integers $p$, $q$ with $2\leq p<q$ and $l=pq$.
We define $R_{l, q}$ as the subgroup of $R_l$ generated by $E_l$ and all elements of the form $1/q^n$ with $n\in \N$.
Similarly, we define $R_{l, p}$ as the subgroup of $R_l$ generated by $E_l$ and all elements of the form $1/p^m$ with $m\in \N$.
Each element $x$ of $R_{l, q}\setminus E_l$ is uniquely written as the formal sum
\[x=x_nq^n+x_{n+1}q^{n+1}+\cdots +x_{-1}q^{-1}+\sum_{j=0}^{\infty}x_jl^j\]
with $n\in \Z_-$; $x_i\in \{ 0, 1,\ldots, q-1\}$ for any integer $i$ with $n\leq i\leq -1$; $x_n\neq 0$; and $x_j\in \{ 0, 1,\ldots, l-1\}$ for any $j\in \N$.
Since $p$ and $q$ are coprime, the element $x$ is also uniquely written as the formal sum
\[x=y_n\left( \frac{q}{p}\right)^n+y_{n+1}\left( \frac{q}{p}\right)^{n+1}+\cdots +y_{-1}\left( \frac{q}{p}\right)^{-1}+\sum_{j=0}^{\infty}y_jl^j\]
with $y_i\in \{ 0, 1,\ldots, q-1\}$ for any integer $i$ with $n\leq i\leq -1$; $y_n\neq 0$; and $y_j\in \{ 0, 1,\ldots, l-1\}$ for any $j\in \N$.
The coefficients $y_n, y_{n+1},\ldots$ are inductively determined by $x_n, x_{n+1},\ldots$, and the difference $\sum_{j=0}^{\infty}x_jl^j-\sum_{j=0}^{\infty}y_jl^j$ lies in $E$.
We also have similar expansions of elements of $R_{l, p}\setminus E_l$.


\section{Actions of certain solvable quotients}\label{sec-sol}

Throughout this section, we fix two coprime integers $p$, $q$ with $1\leq p<q$.
We set $G_0=G_0(p, q)=\Z[1/p, 1/q]$.
Let $a$ denote the multiplicative unit $1$ in $G_0$.
Let $t$ denote the automorphism of the group $G_0$ multiplying by $q/p$.
We define a discrete group $G=G(p, q)$ as the semi-direct product of $G_0$ and the infinite cyclic group generated by $t$.
Note that for any positive integer $r$, the group $G$ is a quotient of the group $\bs(rp, rq)$.
The aim of this section is to construct an interesting ergodic p.m.p.\ action $G\c (Y, \nu)$, which will be used in the next section.

\subsection{The case of $p=1$}\label{subsec-y-1}

We assume $p=1$.
The group $G$ then has the presentation $\langle\, a, t\mid tat^{-1}=a^q\,\rangle$.
We set
\[Y=\prod_{\Z}\{ 0, 1,\ldots, q-1\}\]
and define a probability measure $\nu$ on $Y$ as the direct product of the uniformly distributed probability measure on $\{ 0, 1,\ldots, q-1\}$.
Let $a$ act on $Y$ by the odometer adding $1$ to the $0$th coordinate and increasing digits toward positive coordinates.
More precisely, for each $y=(y_n)_{n\in \Z}\in Y$, the element $ay=(z_n)_{n\in \Z}$ is determined by the formula $z_n=y_n$ for $n\in \Z_-$ and the equality
\[1+\sum_{n=0}^{\infty}y_nq^n=\sum_{n=0}^{\infty}z_nq^n\]
in the group $E_q$.
Let $t$ act on $Y$ by the shift toward the right.
Namely, for each $y=(y_n)_{n\in \Z}\in Y$, the element $ty=(w_n)_{n\in \Z}$ is determined by the formula $w_n=y_{n-1}$ for $n\in \Z$.
This defines an ergodic p.m.p.\ action of $G$ on $(Y, \nu)$.


\subsection{The case of $p>1$}

We assume $p>1$.
The group $G$ then has the presentation so that generators are $a_i$ for $i\in \Z$ and $t$, and relations are $[a_i, a_j]=e$, $a_{i+1}^p=a_i^q$ and $ta_it^{-1}=a_{i+1}$ for any $i, j\in \Z$.
Note that for each $i\in \Z$, $a_i$ corresponds to the number $(q/p)^i\in \Z[1/p, 1/q]$.
Under this identification, we have $a_0=a$.

We set
\[Y_-=\prod_{\Z_-}\{ 0, 1,\ldots, q-1\},\quad Y_0=\prod_{\N}\{ 0, 1,\ldots, pq-1\},\quad Y_+=\prod_{\Z_+}\{ 0, 1,\ldots, p-1\}.\]
We define a probability measure $\nu_-$ on $Y_-$ as the direct product of the uniformly distributed probability measure on $\{ 0, 1,\ldots, q-1\}$.
Similarly, we define probability measures $\nu_0$ and $\nu_+$ on $Y_0$ and $Y_+$, respectively.
We set
\[(Y, \nu)=(Y_-, \nu_-)\times (Y_0, \nu_0)\times (Y_+, \nu_+).\]
The set $Y_0$ is identified with the group $E_{pq}$ under the map sending each element $(x_k)_{k\in \N}$ of $Y_0$ to the sum $\sum_{k\in \N}x_k(pq)^k$.
In the following argument, it is convenient to regard each element of $Y$,
\[y=((y_n)_{n\in \Z_-}, y_0, (y_m)_{m\in \Z_+})\in Y_-\times Y_0\times Y_+,\]
as the formal sum
\[\cdots +y_{-2}\left(\frac{q}{p}\right)^{-2}+y_{-1}\left(\frac{q}{p}\right)^{-1}+y_0+y_1\frac{q}{p}+y_2\left(\frac{q}{p}\right)^2+\cdots.\]
We now define elements $a_iy\in Y$ for $i\in \Z$ and $ty\in Y$.

Let $a_0$ act on $Y$ by adding $1$ on $Y_0$, that is, set
\[a_0y=((y_n)_{n\in \Z_-}, 1+y_0, (y_m)_{m\in \Z_+}),\]
where $Y_0$ is identified with the group $E_{pq}$.
For each $i\in \Z_-$, let $a_i$ act on $Y$ by adding $(q/p)^i$ to $y$.
More precisely, the element
\[a_iy=((z_n)_{n\in \Z_-}, z_0, (z_m)_{m\in \Z_+})\in Y_-\times Y_0\times Y_+\]
of $Y$ is determined by the formula $z_n=y_n$ for any $n\in \Z$ with $n\leq i-1$ or $n\geq 1$ and the equation
\[\left(\frac{q}{p}\right)^i+\sum_{n=i}^{-1}y_n\left(\frac{q}{p}\right)^n+y_0=\sum_{n=i}^{-1}z_n\left(\frac{q}{p}\right)^n+z_0\]
in the group $R_{pq}$.
Similarly, for each $j\in \Z_+$, let $a_j$ act on $Y$ by adding $(q/p)^j$ to $y$.
More precisely, the element
\[a_jy=((w_n)_{n\in \Z_-}, w_0, (w_m)_{m\in \Z_+})\in Y_-\times Y_0\times Y_+\]
of $Y$ is determined by the formula $w_n=y_n$ for any $n\in \Z$ with $n\leq -1$ or $n\geq j+1$ and the equation
\[y_0+\sum_{m=1}^jy_m\left(\frac{q}{p}\right)^m+\left(\frac{q}{p}\right)^j=w_0+\sum_{m=1}^jw_m\left(\frac{q}{p}\right)^m\]
in the group $R_{pq}$.
One can check that for any $i\in \Z$, the Borel automorphism $a_i$ of $Y$ defined above preserves the measure $\nu$, and that for any $i, j\in \Z$, the relations $[a_i, a_j]=e$ and $a_{i+1}^p=a_i^q$ hold.

Let $\eta \colon Y_0\to \{ 0, 1,\ldots, p-1\}$ be the Borel map defined so that for any $x\in E_{pq}$, the element $x-\eta(x)$ belongs to $pE_{pq}$.
We define $ty\in Y$ by the formula
\[ty=((y_{n-1})_{n\in \Z_-},\ y_{-1}+\frac{q}{p}(y_0-\eta(y_0)),\ (\eta(y_0), y_1, y_2,\ldots \ )).\]

\begin{lem}
The map $t\colon Y\to Y$ defined above is a Borel automorphism of $Y$ and preserves the measure $\nu$.
\end{lem}

\begin{proof}
Let $\zeta \colon Y_0\to \{ 0, 1,\ldots, q-1\}$ be the Borel map defined so that for any $x\in E_{pq}$, the element $x-\zeta(x)$ belongs to $qE_{pq}$.
The inverse of the map $t$ is the map sending each element $y$ of $Y$, written as
\[y=((y_n)_{n\in \Z_-}, y_0, (y_m)_{m\in \Z_+})\in Y_-\times Y_0\times Y_+,\]
to
\[((\ \ldots, y_{-2}, y_{-1}, \zeta(y_0)),\ \frac{p}{q}(y_0-\zeta(y_0))+y_1,\ (y_{m+1})_{m\in \Z_+})\in Y_-\times Y_0\times Y_+.\]
The map $t$ is thus a Borel automorphism of $Y$.

Let $N$ and $M$ be positive integers.
Pick
\[k_{-N}, k_{-N+1},\ldots, k_{-1}\in \{ 0, 1,\ldots, q-1\}\quad \textrm{and}\quad l_1, l_2,\ldots, l_M\in \{ 0, 1,\ldots, p-1\}.\]
We also pick $i\in \{ 0, 1,\ldots, p-1\}$ and a Borel subset $A$ of $i+pE_{pq}$.
We define the Borel subset of $Y$,
\begin{align*}
B=\{\, ((y_n)_n, y_0, (y_m)_m)\in Y\mid y_n&=k_n,\ \forall n\in \{ -N, -N+1,\ldots, -1\},\\
y_0&\in A,\ y_m=l_m,\ \forall m\in \{ 1, 2,\ldots, M\}\,\}.
\end{align*}
We then have $\nu(B)=q^{-N}\nu_0(A)p^{-M}$ and
\begin{align*}
tB&=\{\, ((y_n)_n, y_0, (y_m)_m)\in Y\mid y_n=k_{n-1},\ \forall n\in \{ -N+1, -N+2,\ldots, -1\},\\
&y_0\in k_{-1}+\frac{q}{p}(-i+A),\ y_1=i,\ y_m=l_{m-1},\ \forall m\in \{ 2, 3,\ldots, M+1\}\,\}.
\end{align*}
The equality $\nu(tB)=q^{-N+1}(p/q)\nu_0(A)p^{-1}p^{-M}=\nu(B)$ thus holds.
It follows that $t$ preserves $\nu$.
\end{proof}

One can directly check the relation $ta_it^{-1}=a_{i+1}$ as Borel automorphisms of $Y$ for any $i\in \Z$.
We therefore obtain a p.m.p.\ action of $G$ on $(Y, \nu)$.

Let $\mathcal{S}_0$ denote the equivalence relation on $(Y, \nu)$ defined as follows:
Two elements of $Y$,
\[y=((y_n)_{n\in \Z_-}, (\bar{y}_i)_{i\in \N}, (y_m)_{m\in \Z_+})\quad \textrm{and}\quad z=((z_n)_{n\in \Z_-}, (\bar{z}_i)_{i\in \N}, (z_m)_{m\in \Z_+}),\]
are equivalent in $\mathcal{S}_0$ if and only if $y_n=z_n$ for all but finitely many $n\in \Z_-$, $\bar{y}_i=\bar{z}_i$ for all but finitely many $i\in \N$, and $y_m=z_m$ for all but finitely many $m\in \Z_+$.
For a.e.\ $y\in Y$, the equivalence class of $y$ with respect to $\mathcal{S}_0$ is equal to the orbit of $y$ under the action of the subgroup $G_0$ on $Y$.
The action of $G_0$ on $(Y, \nu)$ is therefore ergodic.

\begin{rem}\label{rem-p1}
In the above construction of the action $G\c (Y, \nu)$, if we substitute $1$ for $p$, then $Y_+$ consists of a single point, and the action coincides with the action constructed in Section \ref{subsec-y-1}.
We however decided to discuss the cases of $p=1$ and $p>1$ individually to avoid confusion.
\end{rem}

Let $\theta \colon Y\to Y_0$ denote the projection.
The following lemma will be used in the proof of Lemma \ref{lem-2-ai}.

\begin{lem}\label{lem-theta}
For any $s\in G_0$ and a.e.\ $y\in Y$, there exists a number $K=K(s, y)\in \N$ such that for any $k\in \N$ with $k\geq K$, we have $\theta(t^{-k}sy)=\theta(t^{-k}y)$.
\end{lem}

\begin{proof}
The product of two elements of the group $G_0$ is denoted multiplicatively.
Suppose that the lemma is true for $s_1, s_2\in G_0$.
For a.e.\ $y\in Y$ and any integer $k$ bigger than $K(s_1, s_2y)$ and $K(s_2, y)$, the equality
\[\theta(t^{-k}(s_1s_2)y)=\theta(t^{-k}s_1(s_2y))=\theta(t^{-k}s_2y)=\theta(t^{-k}y)\]
holds.
The lemma is therefore true for $s_1s_2$.

It is enough to show the lemma when $s=t^iat^{-i}=(q/p)^i$ for some $i\in \Z$ because all elements of this form generate $G_0$.
If $i$ is non-negative, then for any integer $k$ and any $y\in Y$, we have the equalities $t^{-k}sy=t^{-k+i+1}(t^{-i-1}st^{i+1})(t^{-i-1}y)$ and $t^{-i-1}st^{i+1}=(q/p)^{-1}$.
It is therefore enough to show the lemma when $s=(q/p)^i$ with $i\in \Z_-$.

For $n\in \Z_-$, let $\rho_n\colon Y\to \{ 0, 1,\ldots, q-1\}$ denote the projection onto the $n$th coordinate of $Y_-$.
For $m\in \N$, we define a Borel map $\tau_m\colon Y\to \{ 0, 1,\ldots, q-1\}$ by $\tau_m(y)=\rho_{-1}(t^{-m-1}y)$ for $y\in Y$.
We have the equality $\rho_n(t^{-m+n}y)=\rho_{-1}(t^{-m-1}y)$ for any $m\in \N$, $n\in \Z_-$ and $y\in Y$.

\begin{claim}\label{claim-a-zero}
The Borel subset of $Y$,
\[A=\{\, y\in Y\mid \tau_m(y)\geq q-p,\ \forall m\in \N \,\},\]
has zero $\nu$-measure.
\end{claim}

\begin{proof}
For $k\in \N$, we set
\[A_k=\{\, y\in Y\mid \tau_m(y)\geq q-p,\ \forall m\in \{ 0, 1,\ldots, k\}\,\}.\]
For any $k\in \N$ and any $y=((y_n)_n, y_0, (y_m)_m)\in Y$, we have
\[t^{-k-1}y=\cdots +y_{-1}\left(\frac{q}{p}\right)^{-k-2}+\sum_{n=-k-1}^{-1}\tau_{n+k+1}(y)\left(\frac{q}{p}\right)^n+\cdots.\]
The inclusion
\[t^{-k-1}A_k\subset \bigcap_{n=-k-1}^{-1}\rho_n^{-1}(\{ q-p, q-p+1,\ldots, q-1\})\]
thus holds.
It follows that $\nu(A_k)=\nu(t^{-k-1}A_k)\leq (p/q)^{k+1}$.
The inclusion $A\subset A_k$ for any $k\in \N$ implies that $\nu(A)=0$.
\end{proof}

\begin{claim}\label{claim-eq}
Fix $n, m\in \Z_+$.
Pick a number $z_l\in \{ 0, 1,\ldots, q-1\}$ indexed by each integer $l$ with $-n-m+1\leq l\leq -n$.
If $z_{-n}\leq q-p-1$, then for each integer $l$ with $-n-m+1\leq l\leq -n$, there exists a unique number $z_l'\in \{ 0, 1,\ldots, q-1\}$ with
\[q\left(\frac{q}{p}\right)^{-n-m}+\sum_{l=-n-m+1}^{-n}z_l\left(\frac{q}{p}\right)^l=\sum_{l=-n-m+1}^{-n}z_l'\left(\frac{q}{p}\right)^l.\]
\end{claim}

\begin{proof}
We find the number $z_l'$ by induction on $m$.
If $m=1$, then the equation
\[q\left(\frac{q}{p}\right)^{-n-1}+z_{-n}\left(\frac{q}{p}\right)^{-n}=(p+z_{-n})\left(\frac{q}{p}\right)^{-n}\]
holds.
It thus suffices to set $z_{-n}'=p+z_{-n}$.
In general, if $p+z_{-n-m+1}\leq q-1$, then we set $z_{-n-m+1}'=p+z_{-n-m+1}$ and $z_l'=z_l$ for any $l$ with $-n-m+2\leq l\leq -n$.
Otherwise, the left hand side of the equation in the claim is equal to
\[(p+z_{-n-m+1}-q)\left(\frac{q}{p}\right)^{-n-m+1}+q\left(\frac{q}{p}\right)^{-n-m+1}+\sum_{l=-n-m+2}^{-n}z_l\left(\frac{q}{p}\right)^l.\]
By using the inequality $0\leq p+z_{-n-m+1}-q\leq q-1$ and the hypothesis of the induction, we can find the number $z_l'$ in the claim.
Uniqueness of $z_l'$ holds because $p$ and $q$ are coprime.
\end{proof}

We now prove Lemma \ref{lem-theta} for $s=(q/p)^i\in G_0$ with $i\in \Z_-$.
Fix $K\in \N$.
For any $k\in \N$ with $k>K$ and any $y=((y_n)_n, y_0, (y_m)_m)\in Y$, we have
\[t^{-k}y=\cdots +\sum_{n=i-k}^{-1-k}y_{n+k}\left(\frac{q}{p}\right)^n+\sum_{m=-k}^{-k+K}\tau_{m+k}(y)\left(\frac{q}{p}\right)^m+\cdots.\]
The equalities $t^{-k}sy=t^{-k}st^k(t^{-k}y)$ and $t^{-k}st^k=(q/p)^{i-k}$ imply that $t^{-k}sy$ is obtained by adding $(q/p)^{i-k}$ to $t^{-k}y$.
If $y_i<q-1$, then the coordinates of $t^{-k}sy$ and $t^{-k}y$ except for the $(i-k)$th one of $Y_-$ are equal.
If $y_i=q-1$ and $\tau_K(y)\leq q-p-1$, then by Claim \ref{claim-eq}, the coordinates of $t^{-k}sy$ and $t^{-k}y$ except for the $j$th one of $Y_-$ with $i-k\leq j\leq -k+K$ are equal.
We thus have $\theta(t^{-k}sy)=\theta(t^{-k}y)$ for any $k\in \N$ with $k>K$ and any $y\in Y$ with $\tau_K(y)\leq q-p-1$.
Lemma \ref{lem-theta} then follows from Claim \ref{claim-a-zero}.
\end{proof}


\section{Stable actions of Baumslag-Solitar groups}\label{sec-s}

Throughout this section, we fix two coprime integers $p$, $q$ with $1\leq p<q$, and fix a positive integer $r$ with $rp\geq 2$.
We set
\[\Gamma =\bs(rp, rq)=\langle\, a, t\mid ta^{rp}t^{-1}=a^{rq}\,\rangle \]
and set $E=\langle a\rangle$.
Let $H$ be the normal subgroup of $\Gamma$ generated by $a$.
The quotient group $\Gamma /H$ is isomorphic to $\Z$ and is generated by the image of $t$.
Let $G_0=G_0(p, q)$ and $G=G(p, q)$ be the groups defined in the beginning of Section \ref{sec-sol}.
We have the surjective homomorphism $\epsilon \colon \Gamma \to G$ sending $a$ to the multiplicative unit $1$ of $G_0$ and sending $t$ to the automorphism of $G_0$ multiplying by $q/p$. 
Let $N$ be the kernel of $\epsilon$.

We set
\[X=\prod_{\N}\{ 0, 1\}\]
and define a probability measure $\mu$ on $X$ as the direct product of the uniformly distributed probability measure on $\{ 0, 1\}$.
We define an ergodic p.m.p.\ action of $\Gamma$ on $(X, \mu)$ so that $a$ acts on it trivially, and $t$ acts on it by the odometer adding $1$ to the $0$th coordinate of $X$ and increasing digits toward the right.
Note that $H$ acts on $X$ trivially.
For $j\in \N$ and $l_0, l_1,\ldots, l_j\in \{ 0, 1\}$, we set
\[X(l_0, l_1, \ldots, l_j)=\{ (x_n)_{n\in \N}\in X\mid x_0=l_0,\ x_1=l_1,\ldots, x_j=l_j\,\}.\]
Let $G\c (Y, \nu)$ be the action constructed in Section \ref{sec-sol}.
Let $\Gamma$ act on $(Y, \nu)$ through the homomorphism $\epsilon \colon \Gamma \to G$.

\subsection{Co-induced actions}\label{subsec-co}

We set
\[Z_0=\prod_{\N}\{ 0, 1,\ldots, pq-1\}\]
and define a probability measure $\xi_0$ on $Z_0$ as the direct product of the uniformly distributed probability measure on $\{ 0, 1,\ldots, pq-1\}$.
Let $E$ act on $(Z_0, \xi_0)$ by odometers so that $a$ adds $1$ to the $0$th coordinate and increases digits toward the right.
We set
\[(Z, \xi)=\prod_{\Gamma/E}(Z_0, \xi_0)\]
and define an action of $\Gamma$ on $(Z, \xi)$ as the action co-induced from the action $E\c (Z_0, \xi_0)$.
Namely, fixing a section $s\colon \Gamma /E\to \Gamma$ for the canonical map from $\Gamma$ onto $\Gamma/E$, for $\gamma \in \Gamma$ and $f\in Z$, we define an element $\gamma f\in Z$ by the equation
\[(\gamma f)(\alpha)=b^{-1}f(\beta)\]
for $\alpha \in \Gamma /E$, where $b\in E$ and $\beta \in \Gamma /E$ are the unique elements determined by the equation $s(\beta)b=\gamma^{-1}s(\alpha)$.
This indeed defines a p.m.p.\ action of $\Gamma$ on $(Z, \xi)$.
One can check that this action is essentially free, that is, the stabilizer of a.e.\ point of $Z$ for the action is trivial, by using that the action of $E$ on $(Z_0, \xi_0)$ is essentially free.
The construction of co-induced actions appears in \cite[Section 3.4]{gab-ex}.

\begin{lem}\label{lem-co-ind}
The following assertions hold:
\begin{enumerate}
\item For any $g\in \Gamma$, there exist $K, K', L, L'\in \N$ such that $K+L=K'+L'$ and for any $k, l\in \N$, we have $ga^{rp^{K+k}q^{L+l}}g^{-1}=a^{rp^{K'+k}q^{L'+l}}$. 
\item For any sequence $\{ n_k\}_{k\in \N}$ of positive integers and any Borel subset $B$ of $Z$, we have
\[\lim_{k\to \infty}\xi(a^{rn_k(pq)^k}B\bigtriangleup B)=0.\]
This convergence is indeed uniform with respect to the sequence $\{ n_k\}_{k\in \N}$.
Namely, for any $\varepsilon >0$ and any Borel subset $B$ of $Z$, there exists $K\in \N$ such that for any sequence $\{ n_k\}_{k\in \N}$ of positive integers and for any $k\in \N$ with $k\geq K$, we have $\xi(a^{rn_k(pq)^k}B\bigtriangleup B)<\varepsilon$.
\item The action of $N$ on $(Z, \xi)$ is ergodic.
\end{enumerate}
\end{lem}

\begin{proof}
The presentation of $\Gamma$ implies that for any $g\in \Gamma$, there exist $K, L\in \N$ with $ga^{rp^Kq^L}g^{-1}\in E$, which is equal to $a^{rp^{K'}q^{L'}}$ for some $K', L'\in \N$ with $K+L=K'+L'$.
The equality in assertion (i) then follows.

By the definition of odometers, for any sequence $\{ m_k\}_{k\in \N}$ of positive integers and any Borel subset $C$ of $Z_0$, we have
\[\lim_{k\to \infty}\xi_0(a^{m_k(pq)^k}C\bigtriangleup C)=0,\]
and this convergence is uniform with respect to the sequence $\{ m_k\}_{k\in \N}$.
Let $F$ be a finite subset of $\Gamma/E$.
By assertion (i), there exist $K, L\in \N$ satisfying the following:
For any $\alpha \in F$, we have $K_{\alpha}, L_{\alpha}\in \N$ such that $K+L=K_{\alpha}+L_{\alpha}$, and if $\{ n_k\}_{k\in \N}$ is a sequence of positive integers, then for any $k\in \N$ with $k\geq K+L$, the equality
\begin{align*}
a^{-rn_k(pq)^k}s(\alpha)&=a^{-rn_kp^Kq^Lp^{k-K}q^{k-L}}s(\alpha)=s(\alpha)a^{-rn_kp^{K_{\alpha}}q^{L_{\alpha}}p^{k-K}q^{k-L}}\\
&=s(\alpha)a^{-rn_kp^{K_{\alpha}+L}q^{L_{\alpha}+K}(pq)^{k-K-L}}
\end{align*}
holds.
Let $\pi_F\colon Z\to \prod_FZ_0$ be the natural projection.
For any sequence $\{ n_k\}_{k\in \N}$ of positive integers and any Borel subset $D$ of $\prod_FZ_0$, we therefore have
\[\lim_{k\to \infty}\xi(a^{rn_k(pq)^k}\pi_F^{-1}(D)\bigtriangleup \pi_F^{-1}(D))=0\]
uniformly with respect to the sequence $\{ n_k\}_{k\in \N}$.
Assertion (ii) is proved.

The assumption $rp\geq 2$ implies that $N$ is infinite.
Assertion (iii) holds because $N$ acts on $\Gamma /E$ freely.
\end{proof}

We set
\[(W, \omega)=(X, \mu)\times (Y, \nu)\times (Z, \xi)\]
and define a p.m.p.\ action $\Gamma \c (W, \omega)$ as the diagonal action so that for $\gamma \in \Gamma$ and $w=(x, y, z)\in W$, we have $\gamma w=(\gamma x, \gamma y, \gamma z)$.
The action $\Gamma \c (W, \omega)$ is ergodic because so are the actions $N\c (Z, \xi)$, $H/N\c (Y, \nu)$ and $\Gamma /H\c (X, \mu)$.
The action $\Gamma \c (W, \omega)$ is essentially free because so is the action $\Gamma \c (Z, \xi)$.
We will show that the action $\Gamma \c (W, \omega)$ has a non-trivial a.c.\ sequence and is therefore stable.
This is enough to show Theorem \ref{thm-stable} for $\Gamma$ because one can find a $\Gamma$-invariant conull Borel subset of $W$ on which $\Gamma$ acts freely.


\subsection{The case of $p=1$}\label{subsec-s-1}

We assume $p=1$.
Fix a positive integer $M$ with $q^M>r$.
We set
\[S=\prod_{\N}\{ 0, 1,\ldots, q^M-1\}\]
and define a probability measure $\sigma$ on $S$ as the direct product of the uniformly distributed probability measure on $\{ 0, 1,\ldots, q^M-1\}$.
The {\it odometer relation} on $(S, \sigma)$, denoted by $\calr_0$, is defined so that two elements $(c_i)_i, (c_i')_i\in S$ are equivalent if and only if for any sufficiently large $i\in \N$, we have $c_i=c_i'$.

Recall that we have constructed the action of $\Gamma$ on the space
\[Y=\prod_{\Z}\{ 0, 1,\ldots, q-1\}\]
with the probability measure $\nu$.
Let $\Gamma$ act on $X\times Y$ diagonally.
For each $j\in \N$, we define $\theta_j\colon X\times Y\to \{ 0, 1,\ldots, q-1\}$ as the projection onto the $j$th coordinate of $Y$.

We define a Borel map $\pi \colon X\times Y\to S$ as follows.
Pick an element
\[\lambda =((x_n)_{n\in \N}, (y_m)_{m\in \Z})\in X\times Y.\]
For each $j\in \N$, the $j$th coordinate of $\pi(\lambda)\in S$, denoted by $\pi(\lambda)_j$, is defined by
\[\pi(\lambda)_j=\sum_{k=0}^{M-1}\theta_{jM+k}(t^{-x_0-2x_1-\cdots -2^jx_j}\lambda)q^k=\sum_{k=0}^{M-1}y_{jM+k+x_0+2x_1+\cdots +2^jx_j}q^k.\]

\begin{lem}\label{lem-1-pi}
In the above notation, the following assertions hold:
\begin{enumerate}
\item The equality $\pi_*(\mu \times \nu)=\sigma$ holds.
\item For any $j\in \N$, the map $\pi(\cdot)_j\colon X\times Y\to \{ 0, 1,\ldots, q^M-1\}$ is invariant under the restriction of $t$ to the subset of $X\times Y$,
\[(X\setminus X(\underbrace{1,\ldots, 1}_{j+1}))\times Y,\]
that is, for any element $\lambda$ of this subset, we have $\pi(t\lambda)_j=\pi(\lambda)_j$.
\item For a.e.\ $\lambda \in X\times Y$, we have $\pi(\Gamma \lambda)\subset \calr_0\pi(\lambda)$.
\end{enumerate}
\end{lem}

\begin{proof}
Pick $d\in \N$ and $h_0, h_1,\ldots, h_d\in \{ 0, 1,\ldots, q^M-1\}$, and set
\[T=\{\, (c_i)_{i\in \N}\in S \mid c_0=h_0,\ c_1=h_1,\ldots, c_d=h_d\,\}.\]
To prove assertion (i), it suffices to show that $(\mu \times \nu)(\pi^{-1}(T))=q^{-(d+1)M}$.
We set
\[\Omega =X(\underbrace{0,\ldots, 0}_{d+1})\quad \textrm{and}\quad A_j=\left\{ \, (y_n)_{n\in \Z}\in Y\Biggm| \sum_{k=0}^{M-1}y_{jM+k}q^k=h_j\,\right\}\]
for each $j\in \{ 0, 1,\ldots, d\}$.
For each integer $l$ with $0\leq l\leq 2^{d+1}-1$, we denote by $l_0, l_1,\ldots, l_d\in \{ 0, 1\}$ as the numbers determined by the equation $l=l_0+2l_1+\cdots +2^dl_d$.
We then have the equality
\[\pi^{-1}(T)=\bigsqcup_{l=0}^{2^{d+1}-1}(t^l\Omega)\times \left( \bigcap_{j=0}^d t^{l_0+2l_1+\cdots +2^jl_j}A_j\right).\]
Fix an integer $l$ with $0\leq l\leq 2^{d+1}-1$.
For any $j\in \{ 0, 1,\ldots, d\}$, we have
\[t^{l_0+2l_1+\cdots +2^jl_j}A_j=\left\{ \, (y_n)_{n\in \Z}\in Y\Biggm| \sum_{k=0}^{M-1}y_{jM+k+l_0+2l_1+\cdots +2^jl_j}q^k=h_j\,\right\}.\]
The inequality
\[jM+M-1+l_0+2l_1+\cdots +2^jl_j<(j+1)M+l_0+2l_1+\cdots +2^{j+1}l_{j+1}\]
for any $j\in \{ 0, 1,\ldots, d-1\}$ implies that
\[\nu \left( \bigcap_{j=0}^d t^{l_0+2l_1+\cdots +2^jl_j}A_j\right)=q^{-(d+1)M}.\]
It follows that $(\mu \times \nu)(\pi^{-1}(T))=q^{-(d+1)M}$.
Assertion (i) is proved.

Assertion (ii) follows from the definition of $\pi(\cdot)_j$.
Assertion (ii) implies that for a.e.\ $\lambda \in X\times Y$, the element $\pi(t\lambda)$ belongs to $\calr_0\pi(\lambda)$.
The action of $H$ on $(Y, \nu)$ generates the odometer relation on $(Y, \nu)$ so that two elements $(y_n)_n, (y_n')_n\in Y$ are equivalent if and only if for all but finitely many $n\in \Z$, we have $y_n=y_n'$.
The action of $H$ on $(X, \mu)$ is trivial.
By the definition of $\pi$, for a.e.\ $\lambda \in X\times Y$, we have $\pi(H\lambda)\subset \calr_0\pi(\lambda)$.
Assertion (iii) is proved.
\end{proof}

For $j\in \N$, we set
\[C_j=\{\, (c_i)_{i\in \N}\in S\mid c_j=0\,\}\in \calb_S,\quad B_j=\pi^{-1}(C_j)\times Z\in \calb_W.\]
For any $j\in \N$, we then have $\omega(B_j)=\sigma(C_j)=q^{-M}$.
The sequence $\{ C_j\}_{j\in \N}$ is an a.i.\ sequence for the odometer relation $\calr_0$.
It follows from Lemma \ref{lem-ai} and Lemma \ref{lem-1-pi} (i), (iii) that $\{ B_j\}_{j\in \N}$ is an a.i.\ sequence for the action $\Gamma \c (W, \omega)$.
We define $U_j\in [\calr(\Gamma \c W)]$ so that for any $l_0, l_1,\ldots, l_j\in \{ 0, 1\}$, we have
\[U_j=t^{l_0+2l_1+\cdots +2^jl_j}a^{rq^{jM}}t^{-l_0-2l_1-\cdots -2^jl_j}=a^{rq^{jM+l_0+2l_1+\cdots +2^jl_j}}\]
on the subset $X(l_0, l_1,\ldots, l_j)\times Y\times Z$ of $W$.
Since $a$ acts on $X$ trivially, the map $U_j$ is indeed an element of the full group $[\calr(\Gamma \c W)]$.

\begin{lem}\label{lem-1-ac}
In the above notation, the following assertions hold:
\begin{enumerate}
\item For any $j\in \N$, we have $U_ja=aU_j$ on $W$, and $U_jt=tU_j$ on the subset of $W$,
\[(X\setminus X(\underbrace{1,\ldots, 1}_{j+1}))\times Y\times Z.\]
\item For any $j\in \N$, we have $U_jB_j\cap B_j=\emptyset$.
\end{enumerate}
\end{lem}

\begin{proof}
Assertion (i) follows from the definition of $U_j$.
For any $j\in \N$, we have
\[B_j=\bigsqcup_{l=0}^{2^{j+1}-1}\left\{ \, ((x_n)_n, (y_m)_m, z)\in W \Biggm| \sum_{n=0}^j2^nx_n=l,\ \sum_{k=0}^{M-1}y_{jM+k+l}q^k=0\, \right\}.\]
The equation $\sum_{k=0}^{M-1}y_{jM+k+l}q^k=0$ is equivalent to the condition that $y_{jM+k+l}=0$ for any $k\in \{ 0, 1,\ldots, M-1\}$.
Since we have chosen the number $M$ so that $q^M>r$ in the beginning of this subsection, for any $j\in \N$, we have
\[U_jB_j=\bigsqcup_{l=0}^{2^{j+1}-1}\left\{ \, ((x_n)_n, (y_m)_m, z)\in W \Biggm| \sum_{n=0}^j2^nx_n=l,\ \sum_{k=0}^{M-1}y_{jM+k+l}q^k=r\, \right\}.\]
Assertion (ii) follows.
\end{proof}

\begin{thm}\label{thm-p-1}
In the above notation, the sequence $\{ U_j\}_{j\in \N}$ is a non-trivial a.c.\ sequence for the action $\Gamma \c (W, \omega)$.
The action $\Gamma \c (W, \omega)$ is therefore stable.
\end{thm}

\begin{proof}
It suffices to check the following three conditions:
\begin{enumerate}
\item[(1)] For any $B\in \calb_W$, we have $\lim_j\omega(U_jB\bigtriangleup B)=0$.
\item[(2)] For any $g\in \Gamma$, we have $\lim_j\omega(\{\, w\in W\mid U_jgw\neq gU_jw\,\})=0$.
\item[(3)] The sequence $\{ B_j\}_{j\in \N}$ in $\calb_W$ is an a.i.\ sequence for the action $\Gamma \c (W, \omega)$, and we have $\omega(U_jB_j\bigtriangleup B_j)=2q^{-M}$ for any $j\in \N$.
\end{enumerate}
The element $a$ of $\Gamma$ acts on $(X, \mu)$ trivially.
Since $a$ acts on $(Y, \nu)$ by the $q$-adic odometer, for any sequence $\{ n_k\}_{k\in \N}$ of positive integers and any $A\in \calb_Y$, we have
\[\lim_{k\to \infty}\nu(a^{n_kq^k}A\bigtriangleup A)=0.\]
The convergence is uniform with respect to the sequence $\{ n_k\}_{k\in \N}$.
Combining this with Lemma \ref{lem-co-ind} (ii), we obtain condition (1).
Condition (2) follows from Lemma \ref{lem-1-ac} (i).
The former assertion in condition (3) has already been checked right after the proof of Lemma \ref{lem-1-pi}.
The latter assertion follows from Lemma \ref{lem-1-ac} (ii).
\end{proof}


\subsection{The case of $p>1$}\label{subsec-s-2}

We assume $p>1$.
Fix a positive integer $M$ with $(pq)^M>r$.
We set
\[S=\prod_{\N}\{ 0, 1,\ldots, (pq)^M-1\}.\]
We have constructed the action of $\Gamma$ on the standard probability space
\[(Y, \nu)=(Y_-, \nu_-)\times (Y_0, \nu_0)\times (Y_+, \nu_+),\]
where
\[Y_-=\prod_{\Z_-}\{ 0, 1,\ldots, q-1\},\quad Y_0=\prod_{\N}\{ 0, 1,\ldots, pq-1\},\quad Y_+=\prod_{\Z_+}\{ 0, 1,\ldots, p-1\}.\]
Let $\Gamma$ act on $X\times Y$ diagonally.
For each $j\in \N$, let $\theta_j\colon X\times Y\to \{ 0, 1,\ldots, pq-1\}$ be the projection onto the $j$th coordinate of $Y_0$.

We define a Borel map $\pi \colon X\times Y\to S$ as follows.
Pick an element
\[\lambda =((x_n)_{n\in \N}, y)\in X\times Y.\]
For each $j\in \N$, the $j$th coordinate of $\pi(\lambda)\in S$, denoted by $\pi(\lambda)_j$, is defined by
\[\pi(\lambda)_j=\sum_{k=0}^{M-1}\theta_{d_j+k}(t^{-x_0-2x_1-\cdots -2^jx_j}\lambda)(pq)^k,\]
where we put $d_j=j+2^{j+1}$.
For each $j\in \N$, we set
\[C_j=\{\, (c_i)_{i\in \N}\in S\mid c_j=0\,\}\in \calb_S,\quad B_j=\pi^{-1}(C_j)\times Z\in \calb_W.\]

\begin{lem}\label{lem-2-ai}
In the above notation, the following assertions hold:
\begin{enumerate}
\item For any $j\in \N$, we have $\omega(B_j)=(pq)^{-M}$.
\item The sequence $\{ B_j\}_{j\in \N}$ in $\calb_W$ is an a.i.\ sequence for the action $\Gamma \c (W, \omega)$.
\end{enumerate}
\end{lem}

\begin{proof}
Fix $j_0\in \N$.
We set
\[\Omega =X(\underbrace{0,\ldots, 0}_{j_0+1})\quad \textrm{and}\quad A=\bigcap_{k=0}^{M-1}\{ \, (y_-, (y_i)_{i\in \N}, y_+)\in Y_-\times Y_0\times Y_+ \mid y_{d_{j_0}+k}=0 \, \}.\]
The equality $\pi^{-1}(C_{j_0})=\bigsqcup_{l=0}^{2^{j_0+1}-1}t^l(\Omega \times A)$ then holds.
Assertion (i) follows.

We prove assertion (ii).
By the definition of $\pi$, for any $j\in \N$, we have $\pi(t\lambda)_j=\pi(\lambda)_j$ for any element $\lambda$ of the subset
\[(X\setminus X(\underbrace{1,\ldots, 1}_{j+1}))\times Y\]
of $X\times Y$.
We therefore have $\lim_j \omega(tB_j\bigtriangleup B_j)=0$.

Let $\theta \colon Y\to Y_0$ be the projection.
For $K\in \N$, we set
\[Y(K)=\{\, y\in Y\mid \theta(t^{-k}ay)=\theta(t^{-k}y),\ \forall k\geq K\,\}.\]
Lemma \ref{lem-theta} implies that $Y=\bigcup_{K\in \N}Y(K)$.
Pick $\varepsilon >0$.
Choose a number $K_0\in \N$ with $\nu(Y\setminus Y(K_0))<\varepsilon/2$.
For $j\in \N$, we set
\[X_j=\{ \,(x_n)_{n\in \N}\in X \mid x_0+2x_1+\cdots +2^jx_j\geq K_0\,\}.\]
Choose a number $J\in \N$ such that for any $j\in \N$ with $j\geq J$, we have $\mu(X\setminus X_j)<\varepsilon/2$.
For any $j\in \N$ and any $\lambda \in X_j\times Y(K_0)$, we have $\pi(a\lambda)_j=\pi(\lambda)_j$.
For any $j\in \N$, the inclusion
\[a(\pi^{-1}(C_j)\cap (X_j\times Y(K_0)))\subset \pi^{-1}(C_j)\]
thus holds.
For any $j\in \N$ with $j\geq J$, we have
\begin{align*}
\omega(aB_j\setminus B_j)&=(\mu \times \nu)(a\pi^{-1}(C_j)\setminus \pi^{-1}(C_j))\leq (\mu \times \nu)((X\times Y)\setminus (X_j\times Y(K_0)))\\
&\leq \mu(X\setminus X_j)+\nu(Y\setminus Y(K_0))<\varepsilon.
\end{align*}
We therefore obtain $\lim_j\omega(aB_j\bigtriangleup B_j)=0$.
Assertion (ii) is proved.
\end{proof}

For $j\in \N$, we define $U_j\in [\calr(\Gamma \c W)]$ so that for any $l_0, l_1,\ldots, l_j\in \{ 0, 1\}$, we have
\begin{align*}
U_j&=t^{l_0+2l_1+\cdots +2^jl_j}a^{r(pq)^{d_j}}t^{-l_0-2l_1-\cdots -2^jl_j}\\
&=a^{rp^{d_j-l_0-2l_1-\cdots -2^jl_j}q^{d_j+l_0+2l_1+\cdots +2^jl_j}}
\end{align*}
on the subset $X(l_0, l_1,\ldots, l_j)\times Y\times Z$ of $W$.
Since $a$ acts on $X$ trivially, the map $U_j$ is indeed an element of the full group $[\calr(\Gamma \c W)]$.

\begin{lem}\label{lem-2-ac}
In the above notation, the following assertions hold:
\begin{enumerate}
\item For any $j\in \N$, we have $U_ja=aU_j$ on $W$, and $U_jt=tU_j$ on the subset of $W$,
\[(X\setminus X(\underbrace{1,\ldots, 1}_{j+1}))\times Y\times Z.\]
\item For any $j\in \N$, we have $U_jB_j\cap B_j=\emptyset$.
\end{enumerate}
\end{lem}

\begin{proof}
Assertion (i) follows from the definition of $U_j$.
For any $j\in \N$, we have
\[B_j=\bigsqcup_{l=0}^{2^{j+1}-1}\left\{ \, (\lambda, z)\in W \Biggm| \sum_{n=0}^j2^nx_n=l,\ \sum_{k=0}^{M-1}\theta_{d_j+k}(t^{-l}\lambda)(pq)^k=0\, \right\},\]
where we put $\lambda =((x_n)_{n\in \N}, y)\in X\times Y$.
Since we have chosen the number $M$ so that $(pq)^M>r$ in the beginning of this subsection, for any $j\in \N$, we have
\[U_jB_j=\bigsqcup_{l=0}^{2^{j+1}-1}\left\{ \, (\lambda, z)\in W \Biggm| \sum_{n=0}^j2^nx_n=l,\ \sum_{k=0}^{M-1}\theta_{d_j+k}(t^{-l}\lambda)(pq)^k=r\, \right\}.\]
Assertion (ii) follows.
\end{proof}

\begin{thm}\label{thm-p-2}
In the above notation, the sequence $\{ U_j\}_{j\in \N}$ is a non-trivial a.c.\ sequence for the action $\Gamma \c (W, \omega)$.
The action $\Gamma \c (W, \omega)$ is therefore stable.
\end{thm}

\begin{proof}
It suffices to check the following three conditions:
\begin{enumerate}
\item[(1)] For any $B\in \calb_W$, we have $\lim_j\omega(U_jB\bigtriangleup B)=0$.
\item[(2)] For any $g\in \Gamma$, we have $\lim_j\omega(\{\, w\in W\mid U_jgw\neq gU_jw\,\})=0$.
\item[(3)] The sequence $\{ B_j\}_{j\in \N}$ in $\calb_W$ is an a.i.\ sequence for the action $\Gamma \c (W, \omega)$, and we have $\omega(U_jB_j\bigtriangleup B_j)=2(pq)^{-M}$ for any $j\in \N$.
\end{enumerate}
The element $a$ of $\Gamma$ acts on $(X, \mu)$ trivially.
Since $a$ acts on $(Y_0, \nu_0)$ by the $(pq)$-adic odometer, and acts on both $(Y_-, \nu_-)$ and $(Y_+, \nu_+)$ trivially, for any sequence $\{ n_k\}_{k\in \N}$ of positive integers and any $A\in \calb_Y$, we have
\[\lim_{k\to \infty}\nu(a^{n_k(pq)^k}A\bigtriangleup A)=0.\]
The convergence is uniform with respect to the sequence $\{ n_k\}_{k\in \N}$.
For any $j\in \N$, we have a Borel partition $X=\bigsqcup_{l=0}^{2^{j+1}-1}X_l$ such that for any $l\in \{ 0, 1,\ldots, 2^{j+1}-1\}$, the restriction of $U_j$ to $X_l\times Y\times Z$ is equal to a power of $a$, and its exponent is divisible by $r(pq)^j$ because $d_j-(2^{j+1}-1)>j$.
Combining this with Lemma \ref{lem-co-ind} (ii), we obtain condition (1).
Condition (2) follows from Lemma \ref{lem-2-ac} (i).
The former assertion in condition (3) is proved in Lemma \ref{lem-2-ai} (ii).
The latter assertion follows from Lemma \ref{lem-2-ai} (i) and Lemma \ref{lem-2-ac} (ii).
\end{proof}

\begin{proof}[Proof of Theorem \ref{thm-stable}]
Let $p$, $q$, $r$ and $\Gamma =\bs(rp, rq)$ be the symbols introduced in the beginning of this section.
It is enough to show that for any two integers $k$, $l$ with $rp=|k|$ and $rq=|l|$, the group $\bs(k, l)$ has an ergodic, free, p.m.p.\ and stable action.
We put $\Lambda =\bs(k, l)$.
If $kl$ is positive, then $\Gamma$ and $\Lambda$ are isomorphic, and the desired assertion follows from Theorems \ref{thm-p-1} and \ref{thm-p-2}.
Assume that $kl$ is negative.
There exists an index $2$ subgroup of $\Gamma$ isomorphic to an index $2$ subgroup of $\Lambda$.
In fact, if we have the presentations
\[\Gamma =\langle\, a, t\mid ta^{rp}t^{-1}=a^{rq}\,\rangle,\quad \Lambda=\langle\, b, u\mid ub^ku^{-1}=b^l\,\rangle,\] 
then let $\Gamma_1$ denote the subgroup of $\Gamma$ generated by $a$, $tat^{-1}$ and $t^2$, and let $\Lambda_1$ denote the subgroup of $\Lambda$ generated by $b$, $ubu^{-1}$ and $u^2$.
We then have $[\Gamma :\Gamma_1]=2$ and $[\Lambda :\Lambda_1]=2$.
The homomorphism $\varphi$ from $\Gamma_1$ into $\Lambda_1$ with $\varphi(a)=b$, $\varphi(tat^{-1})=ub^{-1}u^{-1}$ and $\varphi(t^2)=u^2$ is well-defined and is an isomorphism.

Let $\Gamma \c (W, \omega)$ be the action constructed in Theorems \ref{thm-p-1} and \ref{thm-p-2}.
The subset $W_1=X(0)\times Y\times Z$ is $\Gamma_1$-invariant and has $\omega$-measure $1/2$.
We can define an action of $\Lambda$ on $(W, \omega)$ such that the associated equivalence relation is equal to that for the action $\Gamma \c (W, \omega)$; the subset $W_1$ is $\Lambda_1$-invariant; and the actions $\Gamma_1\c W_1$ and $\Lambda_1\c W_1$ are conjugate through the isomorphism $\varphi$.
This action $\Lambda \c (W, \omega)$ is a desired one.
\end{proof}


\subsection{Stable actions of $H$}

Let $p$, $q$, $r$, $\Gamma =\bs(rp, rq)$ and $H$ be the symbols introduced in the beginning of this section.
Let $\Gamma \c (Y, \nu)$ and $\Gamma \c (Z, \xi)$ be the actions constructed in Sections \ref{subsec-co}--\ref{subsec-s-2}.
We set
\[(W_1, \omega_1)=(Y, \nu)\times (Z, \xi)\]
and define an action $H\c (W_1, \omega_1)$ as the diagonal action, which is ergodic, p.m.p.\ and essentially free.

\begin{thm}\label{thm-h-stable}
In the above notation, the action $H\c (W_1, \omega_1)$ is stable.
\end{thm}

We construct a non-trivial a.c.\ sequence for this action.
The argument is similar to and simpler than those in Sections \ref{subsec-s-1} and \ref{subsec-s-2}.

\begin{proof}[Proof of Theorem \ref{thm-h-stable} in the case of $p=1$]
We define the positive integer $M$, the standard probability space $(S, \sigma)$ and the odometer relation $\calr_0$ on $(S, \sigma)$ as in the beginning of Section \ref{subsec-s-1}.
Define a Borel map $\tau \colon Y\to S$ as follows.
Pick an element $y=(y_n)_{n\in \Z}$ of $Y$.
For each $j\in \N$, the $j$th coordinate of $\tau(y)$, denoted by $\tau(y)_j$, is defined by the formula
\[\tau(y)_j=\sum_{k=0}^{M-1}y_{jM+k}q^k.\]
The equality $\tau_*\nu =\sigma$ then holds.
For a.e.\ $y\in Y$, we have $\tau(Hy)\subset \calr_0\tau(y)$ because the action $H\c (Y, \nu)$ generates the odometer relation on $(Y, \nu)$.

For $j\in \N$, let $C_j=\{\, (c_i)_{i\in \N}\in S\mid c_j=0\,\}$ be the Borel subset of $S$ defined in Section \ref{subsec-s-1}.
Set $A_j=\tau^{-1}(C_j)\times Z$.
For any $j\in \N$, we have $\omega_1(A_j)=\sigma(C_j)=q^{-M}$.
The sequence $\{ A_j\}_{j\in \N}$ is an a.i.\ sequence for the action $H\c (W_1, \omega_1)$.
For $j\in \N$, we define an element $V_j$ of $[\calr(H\c W_1)]$ by setting $V_j=a^{rq^{jM}}$ on $W_1$.
The set $A_j$ consists of all points $(y, z)$ of $W_1=Y\times Z$ with $\tau(y)_j=0$.
The set $V_jA_j$ consists of all points $(y, z)$ of $W_1$ with $\tau(y)_j=r$ because we assumed $q^M>r$.
It follows that $V_jA_j\cap A_j=\emptyset$ and $\omega_1(V_jA_j\bigtriangleup A_j)=2q^{-M}$ for any $j\in \N$.

We now check the following three conditions:
\begin{enumerate}
\item[(1)] For any $A\in \calb_{W_1}$, we have $\lim_j\omega_1(V_jA\bigtriangleup A)=0$.
\item[(2)] For any $g\in H$, we have $\lim_j\omega_1(\{\, w\in W_1\mid V_jgw\neq gV_jw\,\})=0$.
\item[(3)] The sequence $\{ A_j\}_{j\in \N}$ in $\calb_{W_1}$ is an a.i.\ sequence for the action $H\c (W_1, \omega_1)$, and we have $\omega_1(V_jA_j\bigtriangleup A_j)=2q^{-M}$ for any $j\in \N$.
\end{enumerate}
Condition (1) follows from Lemma \ref{lem-co-ind} (ii) and that $a$ acts on $(Y, \nu)$ by the $q$-adic odometer.
Condition (2) holds because for any $g\in H$, we have $ga^{rq^k}g^{-1}=a^{rq^k}$ for any sufficiently large, positive integer $k$.
Condition (3) has already been checked in the previous paragraph.
We have shown that $\{ V_j\}_{j\in \N}$ is a non-trivial a.c.\ sequence for the action $H\c (W_1, \omega_1)$.
Theorem \ref{thm-h-stable} in the case of $p=1$ is proved.
\end{proof}

\begin{proof}[Proof of Theorem \ref{thm-h-stable} in the case of $p>1$]
The proof of this case is similar to that in the case of $p=1$.
We hence give only a sketch of the proof.
Let $M$ be a positive integer with $(pq)^M>r$, and define the standard Borel space
\[S=\prod_{\N}\{ 0, 1,\ldots, (pq)^M-1\}\]
as in the beginning of Section \ref{subsec-s-2}.
Let $\sigma$ be the probability measure on $S$ defined as the direct product of the uniformly distributed probability measure on the set $\{ 0, 1,\ldots, (pq)^M-1\}$.
Let $\calr_0$ denote the odometer relation on $(S, \sigma)$ so that two elements $(c_i)_i, (c_i')_i\in S$ are equivalent if and only if for any sufficiently large $i\in \N$, we have $c_i=c_i'$.

Define a Borel map $\tau \colon Y\to S$ as follows.
Pick an element $y=(y_-, (y_i)_{i\in \N}, y_+)$ of $Y=Y_-\times Y_0\times Y_+$.
For each $j\in \N$, the $j$th coordinate of $\tau(y)$, denoted by $\tau(y)_j$, is defined by the formula
\[\tau(y)_j=\sum_{k=0}^{M-1}y_{jM+k}(pq)^k.\]
The equality $\tau_*\nu =\sigma$ then holds.
For a.e.\ $y\in Y$, we have $\tau(Hy)\subset \calr_0\tau(y)$ because the action $H\c (Y, \nu)$ generates the equivalence relation $\mathcal{S}_0$ on $(Y, \nu)$ defined right before Remark \ref{rem-p1}.

For $j\in \N$, we set $C_j=\{\, (c_i)_{i\in \N}\in S\mid c_j=0\,\}$ and $A_j=\tau^{-1}(C_j)\times Z$.
We define an element $V_j$ of $[\calr(H\c W_1)]$ by $V_j=a^{r(pq)^{jM}}$ on $W_1$.
The sequence $\{ A_j\}_{j\in \N}$ is then an a.i.\ sequence for the action $H\c (W_1, \omega_1)$.
We have $V_jA_j\cap A_j=\emptyset$ and $\omega_1(V_jA_j\bigtriangleup A_j)=2(pq)^{-M}$ for any $j\in \N$.
The sequence $\{ V_j\}_{j\in \N}$ can be checked to be a non-trivial a.c.\ sequence for the action $H\c (W_1, \omega_1)$ along the proof in the case of $p=1$.
\end{proof}

\begin{rem}
Let $k$ and $l$ be integers with $rp=|k|$ and $rq=|l|$.
Set $\Lambda =\bs(k, l)=\langle\, b, u\mid ub^ku^{-1}=b^l\,\rangle$.
The normal subgroup of $\Lambda$ generated by $b$ is isomorphic to $H$ and therefore has an ergodic, free, p.m.p.\ and stable action by Theorem \ref{thm-h-stable}.
\end{rem}


\section{Semi-direct products ME to direct products}\label{sec-sdp}

\begin{defn}\label{defn-me}
Two discrete groups $\Gamma$ and $\Lambda$ are called {\it measure equivalent (ME)} if there exists a measure-preserving action of $\Gamma \times \Lambda$ on a standard Borel space $\Sigma$ with a $\sigma$-finite positive measure $m$ such that we have Borel subsets $X$, $Y$ of $\Sigma$ with $m(X)<\infty$, $m(Y)<\infty$ and the equality $\Sigma =\bigsqcup_{\gamma \in \Gamma}(\gamma, e)Y=\bigsqcup_{\lambda \in \Lambda}(e, \lambda)X$ up to $m$-null sets.
\end{defn}

ME is indeed an equivalence relation between discrete groups (see \cite[Section 2]{furman-mer}).
It is known that two discrete groups $\Gamma$ and $\Lambda$ are ME if and only if there exist ergodic, essentially free and p.m.p.\ actions $\Gamma \c (X, \mu)$ and $\Lambda \c (Y, \nu)$ which are weakly orbit equivalent, that is, we have Borel subsets $A\subset X$ and $B\subset Y$ with positive measure such that the two equivalence relations
\[\calr(\Gamma \c X)\cap (A\times A)\quad \textrm{and}\quad \calr(\Lambda \c Y)\cap (B\times B)\]
are isomorphic (see \cite[Section 3]{furman-oer}).
To prove Theorem \ref{thm-s-int}, we need the following:

\begin{lem}\label{lem-sdp}
Let $K$, $L$ and $M$ be discrete groups.
Let $G=(K\times L)\rtimes M$ be a semi-direct product such that for any $m\in M$, we have $m(K\times \{ e\})m^{-1}=K\times \{ e\}$ and $m(\{ e\}\times L)m^{-1}=\{ e\} \times L$.
Suppose that we have a non-increasing sequence of finite index, normal subgroups of $M$, $M=M_0>M_1>M_2>\cdots$, with $K$ equal to the union of the centralizer of $M_n$ in $K$ for $n\in \N$.

We define $H=K\times (L\rtimes M)$ as the direct product of $K$ with the subgroup $L\rtimes M$ of $G$.
Then $G$ and $H$ are ME.
\end{lem}

\begin{proof}
Pick free and p.m.p.\ actions $K\rtimes M\c (X, \mu)$ and $L\rtimes M\c (Y, \nu)$ of the subgroups (or quotients) of $G$.
We define $Z$ as the projective limit $\varprojlim M/M_n$ and define $\xi$ as the normalized Haar measure on the compact group $Z$.
We have the p.m.p.\ action of $M$ on $(Z, \xi)$ defined by left multiplication.
We define an action of $G$ on the direct product $(W, \omega)=(X, \mu)\times (Y, \nu)\times (Z, \xi)$ by the formula
\[((k, l), m)(x, y, z)=((k, m)x, (l, m)y, mz)\]
for $k\in K$, $l\in L$, $m\in M$, $x\in X$, $y\in Y$ and $z\in Z$.
This action of $G$ is free and p.m.p.

For each $n\in \N$, let $K_n$ denote the centralizer of $M_n$ in $K$ as subgroups of $G$.
Let $\pi_n\colon Z\to M/M_n$ be the canonical projection. 
We define an action of $K_n$ on $(W, \omega)$, denoted by $\alpha_n$, by the formula $\alpha_n(k)w=mkm^{-1}w$ for $k\in K_n$, $m\in M$ and $w\in X\times Y\times \pi_n^{-1}(mM_n)$.
This action is well-defined and commutes with the action of $M$ on $(W, \omega)$.
For any $n\in \N$ and any $k\in K_n$, we have $\alpha_{n+1}(k)=\alpha_n(k)$.

By assumption, $K$ is equal to the union $\bigcup_n K_n$.
We can thus define an action of $K$ on $(W, \omega)$, denoted by $\alpha$, by $\alpha(k)w=\alpha_n(k)w$ for $k\in K_n$ and $w\in W$.
This action $\alpha$ of $K$ commutes with the action of the subgroup $L\rtimes M$ on $(W, \omega)$.
We therefore obtain an action of $H$ on $(W, \omega)$.
This action is free and p.m.p., and gives rise to the same equivalence relation as that for the action of $G$ on $(W, \omega)$.
It follows that $G$ and $H$ are ME.
\end{proof}

\begin{proof}[Proof of Theorem \ref{thm-s-int}]
Applying an argument of the same kind as in the proof of Theorem \ref{thm-stable}, we can reduce the proof of Theorem \ref{thm-s-int} to that in the case where $p$ and $q$ are integers with $2\leq p<q$.
Let $p$ and $q$ be such integers.
Let $S$ be a set of prime numbers dividing neither $p$ nor $q$.
We may assume that $S$ is non-empty.
We set $F=\Z_S$ and define $\Lambda$ as the HNN extension of $F$ relative to the isomorphism $\alpha$ from $pF$ onto $qF$ multiplying by $q/p$.
Let $t$ denote the element of $\Lambda$ implementing $\alpha$ with the relation $tbt^{-1}=\alpha(b)$ for any $b\in pF$. 
We will construct a locally compact and second countable group $G$ containing a lattice isomorphic to $\Lambda$.

Let $T$ denote the Bass-Serre tree associated with the decomposition of $\Lambda$ into the HNN extension.
The set of vertices of $T$ is $\Lambda /F$, and the set of edges of $T$ is $\Lambda /(qF)$.
For $\lambda \in \Lambda$, the edge corresponding to $\lambda (qF)\in \Lambda /(qF)$ joins the vertices corresponding to $\lambda F, \lambda tF\in \Lambda /F$.
We introduce an orientation of $T$ so that for each $\lambda \in \Lambda$, the vertex corresponding to $\lambda F$ is the origin of the edge corresponding to $\lambda (qF)$.
The group $\Lambda$ acts on $T$ by left multiplication, as simplicial automorphisms preserving this orientation.
Let $\aut(T)$ denote the group of orientation-preserving simplicial automorphisms of $T$, which is locally compact and second countable.
We have the homomorphism $\imath \colon \Lambda \to \aut(T)$ associated with the action of $\Lambda$ on $T$.
We also have the continuous homomorphism $\pi \colon \aut(T)\to \Z$ with $\pi(\imath(t))=1$ and $\pi(\varphi)=0$ for any element $\varphi$ of $\aut(T)$ fixing a vertex of $T$.

For a prime number $r$, let $Q_r$ denote the quotient group $\mathbb{Q}_r/\Z_r$, which is regarded as a discrete group.
We define $Q$ as the direct sum $\oplus_{r\in S}Q_r$.
For each $r\in S$, since neither $p$ nor $q$ is divisible by $r$, the multiplication by $q/p$ induces an automorphism of $Q_r$.  
We define an action of $\aut(T)$ on $Q\times \R$ by automorphisms, by the formula
\[\varphi((x_r)_{r\in S}, y)=(((q/p)^{\pi(\varphi)}x_r)_{r\in S}, (q/p)^{\pi(\varphi)}y)\]
for $\varphi \in \aut(T)$, $(x_r)_{r\in S}\in Q$ and $y\in \R$.
Let $G$ denote the associated semi-direct product $(Q\times \R )\rtimes \aut(T)$.

We define a homomorphism $\tau \colon \Lambda \to G$ by
\[\tau(b)=((([b]_r)_{r\in S}, b), \imath(b)),\quad \tau(t)=((0, 0), \imath(t))\]
for any $b\in F$, where $[b]_r$ denotes the equivalence class of $b$ in $Q_r$ for $r\in S$.
The homomorphism $\tau$ is injective.
The image $\tau(\Lambda)$ is a lattice in $G$ with $\{ 0\} \times [0, 1)\times K_0$ a fundamental domain, where $K_0$ is the stabilizer in $\aut(T)$ of the vertex of $T$ whose stabilizer in $\Lambda$ is equal to $F$.

Let $E$ be the infinite cyclic subgroup of $F$ generated by the multiplicative unit of the ring $F$.
We define $\Gamma$ as the subgroup of $\Lambda$ generated by $E$ and $t$.
The group $\Gamma$ is the HNN extension generated by $E$ and $t$ and whose relations are $tbt^{-1}=\alpha(b)$ for any $b\in pE$.
It follows that $\Gamma$ is isomorphic to $\bs(p, q)$.
The image $\tau(\Gamma)$ is a lattice in the subgroup $(\{ 0\} \times \R )\rtimes \aut(T)$ of $G$, with $\{ 0\} \times [0, 1)\times K_0$ a fundamental domain.
Let $\Delta$ be the subgroup of $G$ generated by $Q$ and $\tau(\Gamma)$.
The group $\Delta$ is written as the semi-direct product
\[(Q\times \imath(\ker (\pi \circ \imath)\cap \Gamma))\rtimes \langle \imath(t)\rangle\]
and is a lattice in $G$ with $\{ 0\} \times [0, 1)\times K_0$ a fundamental domain.
It turns out that $\Lambda$ and $\Delta$ are ME.
Since $Q$ is an increasing union of finite subgroups invariant under the action of $\imath(t)$, by applying Lemma \ref{lem-sdp}, we see that $\Delta$ and $Q\times \Gamma$ are ME.
The groups $\Lambda$ and $Q\times \Gamma$ are therefore ME.
By Theorem \ref{thm-stable}, $\Lambda$ and $\Gamma$ are ME.
\end{proof}

We end this section with another application of Lemma \ref{lem-sdp} and a comment on Monod-Shalom's class $\mathscr{C}$.
Let $p$ and $q$ be coprime integers with $1\leq p<q$.
Let $r$ be a positive integer with $rp\geq 2$.
Define the group $G=G(p, q)$ as in the beginning of Section \ref{sec-sol}.
Set
\[\Gamma =\bs(rp, rq)=\langle\, a, t\mid ta^{rp}t^{-1}=a^{rq}\,\rangle.\]
We define the homomorphism $\epsilon \colon \Gamma \to G$ and the subgroups $H$, $N$ of $\Gamma$ as in the beginning of Section \ref{sec-s}.
We set $E=\langle a\rangle$ and define $L$ as the subgroup $\epsilon^{-1}(\epsilon(E))$ of $H$, which is the semi-direct product $N\rtimes \langle a\rangle$.

\begin{prop}\label{prop-l}
Let $F_{\infty}$ denote the free group of countably infinite rank.
Then $L$ and $F_{\infty}\times \Z$ are ME.
\end{prop}

\begin{proof}
The group $N$ is free because it acts on the Bass-Serre tree associated with $\Gamma$ freely.
For $k\in \N$, we denote by $Z_k$ the centralizer of $a^{r(pq)^k}$ in $N$.
We then have the equality $N=\bigcup_{k\in \N}Z_k$.
Since this union is strictly increasing, the group $N$ is not finitely generated.
It turns out that $N$ is isomorphic to $F_{\infty}$.
The proposition follows from Lemma \ref{lem-sdp}.
\end{proof}

\begin{rem}\label{rem-ms}
Monod-Shalom \cite{ms} introduced Class $\mathscr{C}$ consisting of all discrete groups $A$ such that for some mixing unitary representation $\pi$ of $A$ on a Hilbert space, the second bounded cohomology group of $A$ with coefficient $\pi$ is non-zero.
They proved the following three assertions:
\begin{itemize}
\item For any discrete group $A$ in $\mathscr{C}$, any infinite normal subgroup of $A$ belongs to $\mathscr{C}$.
\item Whether a discrete group belongs to $\mathscr{C}$ or not is invariant under ME.
\item No amenable group belongs to $\mathscr{C}$.
\end{itemize}
We refer to \cite[Proposition 7.4]{ms}, \cite[Corollary 7.6]{ms} and \cite[Proposition 7.10 (i)]{ms} for these assertions, respectively.
Combining these results with Proposition \ref{prop-l}, we see that none of $L$, $H$ and $\Gamma$ belongs to $\mathscr{C}$.
Although the above Monod-Shalom's results on $\mathscr{C}$ and Theorem \ref{thm-stable} also imply that $\Gamma$ does not belong to $\mathscr{C}$, the proof of this fact through Proposition \ref{prop-l} is much simpler.
\end{rem}


\section{Stable actions of Vaes groups}\label{sec-v}

The following construction of the group $G$ is a slight generalization of the original one due to Vaes \cite{vaes} (see Remark \ref{rem-v} for the group discussed in \cite{vaes}).

\medskip

\noindent {\bf Construction of a group.}
We follow the notation in \cite{vaes}.
For $n\in \N$, let $H_n$ be a non-trivial finite group, and let $E_n$ be a discrete group.
Let $\Lambda$ be a discrete group acting on $H_n$ by automorphisms for each $n\in \N$.
We denote this action as $\lambda \cdot h$ for $\lambda \in \Lambda$ and $h\in H_n$, using a dot.
Set $K=\oplus_{n\in \N}H_n$.
Let $\Lambda$ act on $K$ diagonally, that is, for $\lambda \in \Lambda$ and $h=(h_n)_{n\in \N}\in K$, we have $\lambda \cdot h=(\lambda \cdot h_n)_{n\in \N}$.
For each $N\in \N$, let $K_N$ be the subgroup of $K$ defined by $K_N=\oplus_{n=N}^{\infty}H_n$.
We set $G_0=K\rtimes \Lambda$ and inductively define a group $G_{N+1}$ as the amalgamated free product
\[G_{N+1}=G_N\ast_{K_N}(K_N\times E_N),\]
where $K_N$ is regarded as a subgroup of $G_N$ through the inclusion $K_N<K<G_0<G_N$, and $G_N$ is naturally a subgroup of $G_{N+1}$.
Let $G$ denote the inductive limit of the increasing sequence of groups, $G_0<G_1<G_2<\cdots$.

For any $N\in \N$, we have the homomorphism $\epsilon_N\colon G_{N+1}\to G_N$ that is the identity on $G_N$ and sends $E_N$ to the neutral element.
We thus obtain the homomorphism $\epsilon \colon G\to G_0$ that is equal to $\epsilon_0\circ \epsilon_1\circ \cdots \circ \epsilon_N$ on $G_{N+1}$ for any $N\in \N$.
Let $\delta \colon G\to \Lambda$ denote the composition of $\epsilon$ and the quotient map from $G_0$ onto $\Lambda$.

\medskip

\noindent {\bf Construction of an action.}
The construction is similar to that in Section \ref{sec-s}.
For each $n\in \N$, let $\Lambda_n$ denote the subgroup of $\Lambda$ consisting of all elements acting on $H_i$ trivially for any $i\in \N$ with $i\leq n$.
We define $X$ as the projective limit $\varprojlim \Lambda /\Lambda_n$ and define $\mu$ as the normalized Haar measure on the compact group $X$.
We have the action $\Lambda \c (X, \mu)$ defined by left multiplication.
Let $G$ act on $(X, \mu)$ through the homomorphism $\delta \colon G\to \Lambda$.

We set $Z_0=\prod_{n\in \N}H_n$ and denote by $\xi_0$ the normalized Haar measure on the compact group $Z_0$.
Let $K$ act on $Z_0$ by left multiplication.
We set
\[(Z, \xi)=\prod_{G/K}(Z_0, \xi_0)\]
and define a p.m.p.\ action of $G$ on $(Z, \xi)$ as the action co-induced from the action of $K$ on $(Z_0, \xi_0)$.
Since the action $K\c (Z_0, \xi_0)$ is essentially free, so is the action $G\c (Z, \xi)$.
If $\ker \epsilon$ is infinite, then the action $\ker \epsilon \c (Z, \xi)$ is ergodic because $\ker \epsilon$ acts on $G/K$ freely.

We set $(Y, \nu)=(Z_0, \xi_0)$.
The group $\Lambda$ acts on the group $Y$ by automorphisms, and $K$ acts on $Y$ by left multiplication.
We then obtain the action $G_0\c (Y, \nu)$. 
Let $G$ act on $(Y, \nu)$ through the homomorphism $\epsilon \colon G\to G_0$.

We set
\[(W, \omega)=(X, \mu)\times (Y, \nu)\times (Z, \xi)\]
and define a p.m.p.\ action $G\c (W, \omega)$ as the diagonal action so that for $g\in G$ and $w=(x, y, z)\in W$, we have $gw=(gx, gy, gz)$.
The action $G\c (W, \omega)$ is essentially free, and is ergodic if $\ker \epsilon$ is infinite.

\begin{thm}\label{thm-s-v}
In the above notation, we assume that $\ker \epsilon$ is infinite.
Then the action $G\c (W, \omega)$ is stable.
\end{thm}

\begin{proof}
For each $n\in \N$, let $\tau_n\colon X\to \Lambda /\Lambda_n$ be the canonical projection.
We define a Borel map $\pi \colon W\to Y$ as follows.
Pick $w=(x, y, z)\in W$ and $n\in \N$.
Choosing $\lambda \in \Lambda$ with $x\in \tau_n^{-1}(\lambda \Lambda_n)$, we define the $n$th coordinate of $\pi(w)$, denoted by $\pi(w)_n$, as $\pi(w)_n=\lambda^{-1}\cdot y_n$, where $y_n$ denotes the $n$th coordinate of $y$.
This is well-defined because $\Lambda_n$ acts on $H_n$ trivially.
We can check the following:
\begin{enumerate}
\item The equality $\pi_*\omega =\nu$ holds.
\item The map $\pi \colon W\to Y$ is $\Lambda$-invariant and $(\ker \epsilon)$-invariant.
\item For any $w\in W$, we have $\pi(Kw)=K\pi(w)$.
\end{enumerate}

We now construct a non-trivial a.c.\ sequence for the action $G\c (W, \omega)$.
For each $n\in \N$, fix a non-neutral element $h_n$ of $H_n$.
We choose a subset $I_n$ of $H_n$ with $h_nI_n\cap I_n=\emptyset$ and $|I_n|\geq |H_n|/3$.
Such an $I_n$ is obtained as follows.
Let $M$ be the least integer with $M\geq |H_n|/3$.
We define elements of $H_n$, $l_1,\ldots, l_M$, inductively.
Pick an arbitrary element $l_1$ of $H_n$.
Let $m$ be a positive integer with $m<M$.
If $l_1,\ldots, l_m$ are defined, then we set $L_m=\{ l_1,\ldots, l_m\}$.
The set
\[H_n\setminus (h_n^{-1}L_m\cup L_m\cup h_nL_m)\]
is non-empty because $m<|H_n|/3$.
Let $l_{m+1}$ be an element of this set.
We defined $l_1,\ldots, l_M\in H_n$, and set $I_n=\{ l_1,\ldots, l_M\}$.
This is a desired set.

We set
\[C_n=\{\, (y_m)_{m\in \N}\in Y\mid y_n\in I_n\,\} \in \calb_Y,\quad B_n=\pi^{-1}(C_n)\in \calb_W.\]
Define an element $U_n$ of $[\calr(G\c W)]$ by $U_n=\lambda h_n\lambda^{-1}$ on $\tau_n^{-1}(\lambda \Lambda_n)\times Y\times Z$ with $\lambda \in \Lambda$.
This is well-defined because $\Lambda_n$ acts on $H_n$ trivially.
We check the following three conditions:
\begin{enumerate}
\item[(1)] For any $A\in \calb_W$, we have $\lim_n\omega(U_nA\bigtriangleup A)=0$.
\item[(2)] For any $g\in G$, we have $\lim_n\omega(\{\, w\in W\mid U_ngw\neq gU_nw\,\})=0$.
\item[(3)] The sequence $\{ B_n\}_{n\in \N}$ in $\calb_W$ is an a.i.\ sequence for the action $G\c (W, \omega)$, and we have $\omega(U_nB_n\bigtriangleup B_n)\geq 2/3$ for any $n\in \N$.
\end{enumerate}
Let $\{ \gamma_n\}_{n\in \N}$ be a sequence with $\gamma_n\in H_n$ for any $n\in \N$.
Pick $g\in G$ and choose $N\in \N$ with $g\in G_N$.
For any $n\in \N$ with $n\geq N-1$, we have $g^{-1}\gamma_ng\in H_n$.
We therefore have $\lim_n\xi(\gamma_nD\bigtriangleup D)=0$ for any $D\in \calb_Z$.
By the definition of the action of $K$ on $Y$, we have $\lim_n\nu(\gamma_nC\bigtriangleup C)=0$ for any $C\in \calb_Y$.
These convergences are uniform with respect to the sequence $\{ \gamma_n\}_{n\in \N}$.
The action of $K$ on $X$ is trivial.
Condition (1) follows.

For any $n\in \N$, by definition, $U_n$ commutes with any element of $\Lambda$.
For any $m\in \N$, if $n\in \N$ is bigger than $m$, then $U_n$ commutes with any element of $H_m$ and $E_m$.
Since $G$ is generated by $\Lambda$, $H_m$ and $E_m$ for all $m\in \N$, condition (2) follows.

The sequence $\{ C_n\}_{n\in \N}$ in $\calb_Y$ is an a.i.\ sequence for the action $K\c (Y, \nu)$.
Conditions (i)--(iii) and Lemma \ref{lem-ai} imply that $\{ B_n\}_{n\in \N}$ is an a.i.\ sequence for the action $G\c (W, \omega)$.
For any $n\in \N$, we have $\omega(U_nB_n\bigtriangleup B_n)=|h_nI_n\bigtriangleup I_n|/|H_n|\geq 2/3$.
Condition (3) is proved.

Conditions (1)--(3) show that $\{ U_n\}_{n\in \N}$ is a non-trivial a.c.\ sequence for the action $G\c (W, \omega)$.
The theorem follows from Theorem \ref{thm-js}.
\end{proof}

\begin{rem}\label{rem-v}
Vaes \cite{vaes} showed that $G$ is inner amenable, the conjugacy class of any non-neutral element of $G$ in $G$ is infinite, and the von Neumann algebra of $G$ does not have property Gamma, under the following assumption:
We choose an arbitrary sequence of mutually distinct prime numbers, $\{ p_n\}_{n\in \N}$.
For each $n\in \N$, we set $H_n=(\Z/p_n\Z)^3$ and $E_n=\Z$, and set $\Lambda =SL(3, \Z)$.
The group $\Lambda$ naturally acts on $H_n$ by automorphisms.
Theorem \ref{thm-s-v} therefore implies Theorem \ref{thm-v}.
\end{rem}


\end{document}